\documentclass[12pt]{amsart}
\usepackage{amsmath,amssymb}
\newcommand{\IB}{\mathbb{B}}
\newcommand{\IC}{\mathbb{C}}

\newcommand{\IM}{\mathbb{M}}
\newcommand{\IN}{\mathbb{N}}

\newcommand{\IP}{\mathbb{P}}
\newcommand{\IR}{\mathbb{R}}
\newcommand{\IS}{\mathbb{S}}

\newcommand{\cP}{\mathcal{P}}

\newcommand{\ve}{\varepsilon}
\newcommand{\vp}{\varphi}
\newcommand{\Ga}{\Gamma}

\DeclareMathOperator{\supp}{\mathop{supp}}
\DeclareMathOperator{\dist}{dist}
\DeclareMathOperator{\id}{id}
\DeclareMathOperator{\Ball}{Ball}
\DeclareMathOperator{\graph}{graph}
\DeclareMathOperator{\dom}{dom}
\DeclareMathOperator{\ran}{ran}
\DeclareMathOperator{\diag}{diag}

\DeclareMathOperator{\Prob}{Prob}
\newcommand{\ip}[1]{\mathopen{\langle}#1\mathclose{\rangle}}
\newtheorem{thmA}{Theorem}

\newtheorem{corA}[thmA]{Corollary}
\newtheorem{thm}{Theorem}

\newtheorem{lem}[thm]{Lemma}
\theoremstyle{definition}

\theoremstyle{remark}

\title[Uniform Roe algebras and quasi-local algebras]{Embeddings of matrix algebras into uniform Roe algebras and quasi-local algebras}
\author{Narutaka Ozawa}
\address{RIMS, Kyoto University, \mbox{606-8502} Japan}
\email{narutaka@kurims.kyoto-u.ac.jp}
\thanks{The author was partially supported by JSPS KAKENHI Grant Numbers 20H01806 and 20H00114}

\date{\today}

\begin{document}
\begin{abstract}
We answer the recent problem posed by Baudier, Braga, Farah, 
Vignati, and Willett that asks whether the $\ell_\infty$-direct 
sum of the matrix algebras embeds into the uniform Roe algebra 
or the quasi-local algebra of a uniformly locally finite metric space. 
The answers are no and yes, respectively. 
This yields the existence of a quasi-local operator 
that is not approximable by finite propagation operators. 
\end{abstract}
\maketitle
\section{Introduction}
Throughout this paper, we are interested in 
a (discrete) metric space $X$ that is 
\emph{uniformly locally finite}, or \emph{ulf} in short 
(a.k.a.\ of bounded geometry), i.e., 
\[
N_X(R) := \sup\{ |\Ball(x,R)| : x \in X \} < \infty 
\]
for every $R>0$, 
where $\Ball(x,R) := \{ y\in X : \dist(y,x)\le R\}$. 
Associated with $X$ are the 
\emph{uniform Roe algebra} $\mathrm{C}^*_\mathrm{u}[X]$ 
and the \emph{quasi-local algebra} $\mathrm{C}^*_\mathrm{ql}[X]$, 
prototypes of which are introduced in \cite{roe:jdg}. 
These are $\mathrm{C}^*$-sub\-alge\-bras 
of the $\mathrm{C}^*$-alge\-bra $\IB(\ell_2X)$ 
of bounded operators on the Hilbert space $\ell_2X$. 
For $R>0$, we denote by 
\[
\IC_\mathrm{u}^R[X]:=\{ u\in\IB(\ell_2X) : \ip{u\delta_x,\delta_y}=0\mbox{ whenever $\dist(x,y)> R$}\}
\]
the set of all operators with \emph{propagation at most $R$}. 
The uniform Roe algebra $\mathrm{C}^*_\mathrm{u}[X]$ 
is the norm closure of the $*$-algebra $\bigcup_{R>0}\IC_\mathrm{u}^R[X]$ 
of finite propagation operators on $\ell_2X$.
An operator $u$ on $\ell_2X$ is said to have 
\emph{$\ve$-propagation at most $R$} if it satisfies 
$\|1_Au1_B\|\le\ve$ whenever $A,B\subset X$ are such that $\dist(A,B)>R$. 
Here, $1_A \in \IB(\ell_2X)$ stands for the orthogonal projection 
from $\ell_2X$ onto $\ell_2A$ for any $A\subset X$. 
An operator $u\in\IB(\ell_2X)$ is \emph{quasi-local} if it has finite 
$\ve$-propagation for all $\ve>0$. 
The quasi-local algebra $\mathrm{C}^*_\mathrm{ql}[X]$ 
is the $\mathrm{C}^*$-alge\-bra consisting of all 
quasi-local operators. 
It is plain to see that 
$\mathrm{C}^*_\mathrm{u}[X] \subset \mathrm{C}^*_\mathrm{ql}[X]$. 
The uniform Roe algebras and the quasi-local algebras 
have different advantages. 
Generally speaking, an operator in $\mathrm{C}^*_\mathrm{u}[X]$ 
is easier to handle than that in $\mathrm{C}^*_\mathrm{ql}[X]$, 
but it is harder to tell if a given operator $u$ belongs 
to $\mathrm{C}^*_\mathrm{u}[X]$. 
Thus the problem whether they coincide or not 
has caught considerable attention 
(p.\ 20 in \cite{roe:cbms}, see also 
\cite{bbfvw, engel, klvz, lnsz, st, sz} and references therein). 
It is proved in \cite{sz} that a large class of ulf metric spaces, namely those 
with property~A, satisfy the equality 
$\mathrm{C}^*_\mathrm{u}[X] = \mathrm{C}^*_\mathrm{ql}[X]$.
See Section~\ref{sec:propertyA} for the definition of property~A 
and an alternative proof of this fact. 
In this paper, we prove that the inclusion 
$\mathrm{C}^*_\mathrm{u}[X] \subset \mathrm{C}^*_\mathrm{ql}[X]$ 
can be proper in general. 
The proof takes a roundabout way and goes by 
studying the embeddability of 
the $\mathrm{C}^*$-alge\-bra $\prod_n \IM_n$ 
of the $\ell_\infty$-direct sum of matrix algebras. 
Whether embeddings are unital or not will make no essential difference.

\begin{thmA}\label{thm:A}
The $\mathrm{C}^*$-alge\-bra $\prod_n \IM_n$ does not 
embed into the uniform Roe algebra 
$\mathrm{C}^*_\mathrm{u}[X]$ of any ulf metric space $X$. 
\end{thmA}

\begin{thmA}\label{thm:B}
The $\mathrm{C}^*$-alge\-bra $\prod_n \IM_n$ 
embeds into the quasi-local algebra 
$\mathrm{C}^*_\mathrm{ql}[X]$ of a ulf metric space $X$, 
provided that $X$ contains a sequence of expanders. 
\end{thmA}

See Section~\ref{sec:B} for the definition of expanders. 
The above results answer the problem posed in \cite{bbfvw}, 
where it is proved that non-atomic von Neumann 
algebras do not embed into quasi-local algebras, 
leaving the possibility for the atomic von Neumann algebra 
$\prod_n \IM_n$ open. 

\begin{corA}
For any ulf metric space $X$ that contains a sequence 
of expanders, the inclusion 
$\mathrm{C}^*_\mathrm{u}[X] \subset \mathrm{C}^*_\mathrm{ql}[X]$ 
is proper. In other words, there exists a quasi-local operator that 
is not approximable by finite propagation operators. 
\end{corA}

We remark that this corollary holds as well in the 
``non-uniform'' setting (see Chapter~3 in \cite{roe:cbms} 
for the definition and, for a given Hilbert $X$-module $H$, 
consider an $X$-embedding $\ell_2X \subset H$).  
Recall that property~A is a kind of amenability condition and 
a sequence of expanders is the most prominent obstruction to it 
(see e.g., Sections 4 \& 5 in \cite{ny}). 
It seems natural to expect that $\prod_n \IM_n$ embeds into 
the quasi-local algebra and hence 
$\mathrm{C}^*_\mathrm{u}[X] \neq \mathrm{C}^*_\mathrm{ql}[X]$
as soon as $X$ does not have property~A. 
\subsection*{Acknowledgments}
The author is grateful to Professor Ilijas Farah for introducing 
him the problem in \cite{bbfvw} that led him to the present work. 
This research was carried out during the author's stay at 
the Fields Institute for Research in Mathematical Sciences 
for ``Thematic Program on Operator Algebras and Applications'' 
in the Fall 2023. 
The author acknowledges the kind hospitality, the exciting 
environment, and the financial support provided by the institute. 
This research was partially supported by 
JSPS KAKENHI Grant Numbers 20H01806 and 20H00114. 
\section{Proof of Theorem~\ref{thm:A}}\label{sec:A}
The proof of Theorem~\ref{thm:A} is motivated by 
an operator space theoretic perspective that the matrix 
algebras are hard to embed completely isomorphically 
into commutative $\mathrm{C}^*$-alge\-bras. 

For every Banach space $E$, we denote by $(E)_1$ 
the closed unit ball of $E$. 
For every projection $p$, we put $p^\perp := 1-p$. 

\begin{lem}\label{lem:main}
Let $R>0$ and $n$ be such that 
$N_X(R) < \sqrt{n}/8$.
Then for any possibly non-unital embedding 
$\IM_n\hookrightarrow\mathrm{C}^*_\mathrm{u}[X]$ with the unit $p$, 
there is $a\in(\IM_n)_1$ satisfying that 
$\dist(a+b,\IC_\mathrm{u}^R[X])\geq 1/2$
for every $b\in (p^\perp\mathrm{C}^*_\mathrm{u}[X]p^\perp)_1$.
\end{lem}

\begin{proof}[Proof of Theorem~\ref{thm:A}]
Suppose for a contradiction that 
$\prod_n \IM_n\hookrightarrow\mathrm{C}^*_\mathrm{u}[X]$. 
We denote by $p_n \in \mathrm{C}^*_\mathrm{u}[X]$ 
the unit for $\IM_n$. 
Then Lemma~\ref{lem:main} provides 
for each $n$ an element $a_n\in(\IM_n)_1$ that satisfies 
\[
\inf_{b\in (p_n^\perp\mathrm{C}^*_\mathrm{u}[X]p_n^\perp)_1 } \dist(a_n+b,\IC_\mathrm{u}^{R_n}[X]) \geq \frac{1}{2}
\]
for $R_n:=\sup\{R>0 : N_X(R) < \sqrt{n}/8\}-1$.
Notice that $R_n\nearrow\infty$ by uniform local finiteness. 
Now $a:=\diag_n (a_n)_n\in(\prod_n \IM_n)_1$ satisfies 
$\dist(a, \IC_\mathrm{u}^R[X])\geq1/2$ for all $R>0$, 
in contradiction with the hypothesis.
\end{proof}

The rest of this section is devoted for the proof 
of Lemma~\ref{lem:main}. 
The following two lemmas are certainly known to experts, but 
we put their proofs because they are short. 
Recall that a \emph{partial translation} on $X$ is a bijection $T$ 
from $\dom T \subset X$ onto $\ran T \subset X$.  

\begin{lem}\label{lem:graph}
For every $R>0$, there is a family $\{ T_i\}_{i=1}^{2N_X(R)}$ 
of partial translations that satisfies 
$\{ (x,y)\in X : \dist(x,y)\le R \} = \bigsqcup_{i=1}^{2N_X(R)}\graph T_i$. 
\end{lem}
\begin{proof}
We claim that any maximal (w.r.t.\ the graph union) family 
$T_1,\ldots, T_{2N_X(R)}$ of partial translations with mutually 
disjoint graphs does the job. 
Suppose this is not the case and $(x_0,y_0)\notin \bigsqcup\graph T_i$.
Then for each $i$, either $x_0\in\dom T_i$ 
and $T_i(x_0)\in \Ball(x_0,R)\setminus\{y_0\}$ 
or $y_0\in\ran T_i$ and $T_i^{-1}(y_0)\in \Ball(y_0,R)\setminus\{x_0\}$. 
By the pigeonhole principle, this is impossible. 
\end{proof}

\begin{lem}\label{lem:irr}
For every irreducible unitary representation $\pi\colon\Ga\to\IM_n$ 
of a finite group $\Ga$, one has 
\[
\sup_{\alpha\in(\ell_\infty \Ga)_1} \|\frac{1}{|\Ga|}\sum_{g\in\Ga}\alpha_g\pi(g)\|\le\frac{1}{\sqrt{n}}.
\]
\end{lem}
\begin{proof}
For any unit vectors $\xi$ and $\eta$, one has 
\begin{align*}
|\ip{ \frac{1}{|\Ga|}\sum_g \alpha_g\pi(g)\xi,\eta}|
 \le(\frac{1}{|\Ga|}\sum_g |\ip{\pi(g)\xi,\eta}|^2)^{1/2}
 =\ip{ P(\xi\otimes\overline{\xi}),(\eta\otimes\overline{\eta})}^{1/2},
\end{align*}
where $P:=|\Ga|^{-1}\sum_g(\pi\otimes\overline{\pi})(g)$
is the orthogonal projection onto the space of 
$(\pi\otimes\overline{\pi})(\Ga)$ invariant vectors. 
Since $\pi$ is irreducible, by Schur's lemma, 
$\ran P = \IC(n^{-1/2}\sum_i\zeta_i\otimes\overline{\zeta_i})$, where $\{\zeta_i\}$ is any orthonormal basis. 
This implies 
$\ip{ P(\xi\otimes\overline{\xi}),(\eta\otimes\overline{\eta})}=1/n$.
\end{proof}

\begin{proof}[Proof of Lemma~\ref{lem:main}]
Put 
\[
\ve:=\max_{a\in(\IM_n)_1}\inf_{b\in (p^\perp\mathrm{C}^*_\mathrm{u}[X]p^\perp)_1} \dist(a+b,\IC_\mathrm{u}^R[X]) + 1/10.
\] 
Take an irreducible unitary representation 
$\pi\colon\Ga\to\IM_n \subset\mathrm{C}^*_\mathrm{u}[X]$ 
of a finite group $\Ga$ (e.g., the standard representation of 
the symmetric group $\Ga=\mathfrak{S}_{n+1}$) 
and choose for each $g\in\Ga$ elements 
$b_g\in (p^\perp\mathrm{C}^*_\mathrm{u}[X]p^\perp)_1$ 
and $c_g\in\IC_\mathrm{u}^R[X]$ such that 
$\| \pi(g)+b_g-c_g \|<\ve$. 
One has 
\begin{align*}
\|(\frac{1}{|\Ga|}\sum_g c_g\otimes\overline{\pi(g)})\|
 > \|\frac{1}{|\Ga|}\sum_g (\pi(g)+b_g)\otimes\overline{\pi}(g)\|
 - \ve
 \geq 1-\ve.
\end{align*}

Let $\{ (x,y)\in X : \dist(x,y)\le R \} = \bigsqcup_{i=1}^{2N_X(R)}\graph T_i$
by Lemma~\ref{lem:graph} 
and denote by $\Phi_i$ the complete contraction from 
$\mathrm{C}^*_\mathrm{u}[X]$ onto the space of operators supported 
on $\graph T_i$. Also, let $\vp_{x,y}\in(\mathrm{C}^*_\mathrm{u}[X]^*)_1$ denote 
the matrix coefficient corresponding to the $(x,y)$ entry. 
Then, $\sum_{i=1}^{2N_X(R)}\Phi_i$ is the projection 
onto $\IC_\mathrm{u}^R[X]$. 
It follows that there must be $i$ such that 
\[
\|(\Phi_i\otimes\id)(\frac{1}{|\Ga|}\sum_g c_g\otimes\overline{\pi(g)})\| > \frac{1-\ve}{2N_X(R)}.
\]
However, since the range of $\Phi_i$ is completely isometric 
to $\ell_\infty(\graph T_i)$, 
one has 
\[
\|(\Phi_i\otimes\id)(\frac{1}{|\Ga|}\sum_g c_g\otimes\overline{\pi(g)})\|
=\sup_{(x,y)\in\graph T_i}\|\frac{1}{|\Ga|}\sum_g \vp_{x,y}(c_g)\overline{\pi(g)}\|
< \frac{1+\ve}{\sqrt{n}}
\]
by Lemma~\ref{lem:irr}. 
Since $N_X(R)<\sqrt{n}/8$, these inequalities imply $\ve>3/5$. 
\end{proof}

We are grateful to one of the referees for pointing out that 
the above proofs also prove that $\prod_n \IM_n$ does not 
embed into the so-called \emph{maximal uniform Roe algebra} 
nor any other $\mathrm{C}^*$-completion of 
the $*$-algebra $\bigcup_R\IC_{\mathrm{u}}^R[X]$. 
\section{Proof of Theorem~\ref{thm:B}}\label{sec:B}
The proof of Theorem~\ref{thm:B} uses a similar idea 
to \cite{lnsz} and \cite{klvz}.  
Recall that a \emph{sequence of expanders} is a sequence $(X_n)_n$ 
of finite metric spaces (finite graphs in most of the literature) 
such that $|X_n|\to\infty$ and 
\[
\kappa:=\inf_n \min_{\begin{subarray}{c} A\subset X_n;\\ 0<|A|/|X_n|\le1/2\end{subarray}}
 \frac{|\{ x \in X_n : \dist(x,A)\le R\}| }{|A|}  > 1
\]
for some $R>0$. 
It yields that for any $n$ and any subsets $A,B\subset X_n$ 
\[
\min\{ |A|/|X_n|,\,|B|/|X_n| \}\le\kappa^{-\dist(A,B)/2R}.
\] 
Hence, the LHS is arbitrarily small if $\dist(A,B)$ is large enough. 
This property (named \emph{asymptotic expanders} in \cite{lnsz}) 
is what we need in this paper. 
It guarantees that an operator on such a space 
with ``well-spread'' matrix coefficients is quasi-local. 

A new ingredient for constructing quasi-local operators 
is a random projection of rank $n$. 
We use the following model of random $n$-dimensional 
subspaces $V$ in $\IR^d = \ell_2([d],\IR)$. 
Here $[d]:=\{1,\ldots,d\}$. 
The difference between real and complex will not matter; 
if necessary,
we view $V$ as its complexification in the complex Hilbert 
space $\ell_2[d]$. 
We consider the probability spaces 
$\IS^{d-1}:=\{ x\in\IR^d : \|x\|=1\}$ and $(\IS^{d-1})^n$ 
with the probability measures $\IP$. 
For $\mathbf{x}:=(x_1,\ldots,x_n) \in (\IS^{d-1})^n$, 
we put $V(\mathbf{x}):=\mathop{\mathrm{span}}\{ x_1,\ldots,x_n\}$, 
which is $n$-dimensional with probability $1$. 
We write $P_V$ for the orthogonal projection onto $V$. 

\begin{lem}\label{lem:random}
For every $n\in\IN$ and $0<\delta<1/10$, there 
are $c>0$ and $D\in\IN$ that satisfy 
the following property. 
The random $n$-dimensional subspace $V$ 
in $\IR^d$, $d\geq D$, satisfies 
\[
\IP\Bigl( \max_{\begin{subarray}{c} E\subset[d];\\ |E|/d\le\delta\end{subarray}} \| P_V |_{\ell_2E}  \| 
 <100\sqrt{\delta\log(1/\delta)} \Bigr)\geq 1-e^{-c d}.
\]
\end{lem}

\begin{proof}[Proof of Theorem~\ref{thm:B}]
Assume that $X$ contains a sequence $(X_n)_n$ of expanders. 
Put $\delta_n:=1/n$ and $\ve_n:=100\sqrt{\delta_n\log(1/\delta_n)}$. 
For each $n$ find an $n$-dimen\-sional subspace $V(n)$ 
in $\IR^{d(n)}$ that satisfies 
\[
\max\{ \| P_{V(n)} |_{\ell_2E} \| : E\subset[d(n)],\ |E|/d(n)\le\delta_k\}<\ve_k
\]
for all $k=1,\ldots,n$. 
We may assume $|X_n|=d(n)$ and view 
(the complexification of) 
$V(n)$ as a subspace of $\ell_2 X_n\subset\ell_2X$. 
We claim that $\prod_n \IB(V(n))$ is contained in 
$\mathrm{C}^*_{\mathrm{ql}}[X]$. 
Let $u=\diag_n (u_n)_n \in \prod_n \IB(V(n))$ with norm $1$ 
and $\ve>0$ be given arbitrarily. 
Fix $k$ with $\ve_k<\ve$ 
and take $R=R_k>0$ large enough. 
One has to show $\| 1_A u 1_B\|<\ve$ whenever 
$A, B\subset X$ are such that $\dist(A,B)>R$. 
We consider each summand $u_n$ separately. 
Since $R$ is taken large enough, 
$A,B\subset X_n$ with $\dist(A,B)>R$ implies that 
$\min\{|A|/d(n),|B|/d(n)\}<\delta_k$. 
Thus
\[
\|1_A u_n 1_B\|\le \min\{ \|1_A P_{V(n)}\|,\,\| P_{V(n)} 1_B\|\} <\ve_k<\ve
\]
for all $n\geq k$. This proves $u$ is quasi-local. 
\end{proof}

The point of Proof of Theorem~\ref{thm:B} is 
to show the operator $\diag_n (P_{V(n)})_n$ is quasi-local. 
As mentioned in Introduction, it is harder to tell 
if it belongs to $\mathrm{C}^*_{\mathrm{u}}[X]$.  

\begin{proof}[Proof of Lemma~\ref{lem:random}]
We may assume $\ve:=25\sqrt{\delta\log(1/\delta)}\le 1/4$. 
Also for notational simplicity, we assume $d\delta$ is an integer 
and write $\cP(d,\delta):=\{ E\subset [d] : |E|/d=\delta\}$. 
By the measure concentration phenomenon 
(L\'evy's Lemma, see e.g., 14.3.2, 14.3.3, and 15.2.2 in \cite{matousek} 
or Section 2 in \cite{ms}),  
every $E\in \cP(d,\delta)$ satisfies 
\[
\IP( \{ x\in \IS^{d-1} : \|1_E x\|>m_\delta+\ve \} ) < 2e^{-\ve^2 d/2}.
\]
Here $m_\delta$ is the median of $\|1_E x\|$, 
which is at most $\delta^{1/2}+12d^{-1/2}$ (see 14.3.3 in \cite{matousek}). 
It is important that the estimate is uniform in $\delta>0$. 
We have $\ve>m_\delta$ when $d$ is large enough. 
Recall $\log \left(\begin{smallmatrix}d\\ \delta d\end{smallmatrix}\right)\le H(\delta)d$, 
where $H(\delta)=-\delta\log\delta - (1-\delta)\log(1-\delta)$, 
because
$1=(\delta + (1-\delta))^d\geq\left(\begin{smallmatrix}d\\ \delta d\end{smallmatrix}\right)\delta^{\delta d}(1-\delta)^{(1-\delta)d}$. 
We have $H(\delta)<\ve^2/4$. 
Thus
\[
\IP( \{ x\in \IS^{d-1} : \max_{E\in \cP(d,\delta)}
 \| 1_Ex\| >2\ve \} ) < 2e^{-\ve^2 d /4}.
\]

By the measure concentration phenomenon, a random $x\in\IS^{d-1}$ 
is almost orthogonal with high probability to a fixed subspace $W\subset\IR^d$ 
of small dimension. Namely, $\|P_Wx\|^2$ is concentrated around 
its expectation $(\dim W)/d$ when $d$ is large. 
Thus a random $n$-tuple $\mathbf{x}=(x_1,\ldots,x_n) \in (\IS^{d-1})^n$ 
is asymptotically orthonormal as $d\to\infty$. 
This means that for every $\alpha=(\alpha_k)_{k=1}^n\in \IS^{n-1}$, 
the random vector 
$\alpha\cdot\mathbf{x}:=\sum_k \alpha_k x_k$ 
has asymptotically unit norm. 
Moreover, since the distribution of 
$\alpha\cdot\mathbf{x}/\| \alpha\cdot\mathbf{x}\|$ 
is $O(d)$-invariant, one has 
\[
\IP( \{ \mathbf{x}\in (\IS^{d-1})^n : \max_{E\in \cP(d,\delta)} 
\|1_E \alpha\cdot \mathbf{x}\| > 3\ve \} ) < 2e^{-\ve^2 d /4}
\]
for every $\alpha\in \IS^{n-1}$ 
and every $d$ large enough. 
Considering some $\ve$-dense subset in $\IS^{n-1}$, one sees 
\[
\IP( \{ \mathbf{x}\in (\IS^{d-1})^n : \max_{E\in \cP(d,\delta)}
 \|1_E|_{V(\mathbf{x})}\| > 4\ve \} ) < C(n,\ve) e^{-\ve^2 d /4}
\]
for some $C(n,\ve)>0$. 
Because $\|P_V|_{\ell_2E}\| =\|1_E|_V\|$, this proves the lemma.  
\end{proof}
\section{property~A implies $\mathrm{C}^*_{\mathrm{u}}[X]=\mathrm{C}^*_{\mathrm{ql}}[X]$}\label{sec:propertyA}
Let $\Prob(X)\subset\ell_1X$ denote the space of 
probability measures on $X$. 
Recall that $X$ has \emph{property~A} 
(see, e.g., Section~4 in \cite{ny})
if for every $\delta>0$ and $R>0$, 
there are $S>0$ and $\mu\colon X\to\Prob X$ that satisfy 
$\supp \mu_x \subset\Ball(x,S)$ for every $x$ and 
$\|\mu_x - \mu_y \|_1<\delta$ whenever $\dist(x,y)\le R$. 
It is proved in \cite{sz} that 
property~A implies $\mathrm{C}^*_{\mathrm{u}}[X]=\mathrm{C}^*_{\mathrm{ql}}[X]$. 
We give here a more direct and quantitative proof of this fact 
(in the uniform setting; the proof for the ``non-uniform'' 
case is similar, but more bulky). 

\begin{thm}[\cite{sz}]\label{thm:sz}
Let $X$ be a ulf metric space with property~A. 
Then the equality $\mathrm{C}^*_{\mathrm{u}}[X]=\mathrm{C}^*_{\mathrm{ql}}[X]$ holds. 
More precisely, any contraction $u\in\IB(\ell_2X)$ 
with finite $\ve$-propa\-gation satisfies 
$\dist(u,\mathrm{C}^*_{\mathrm{u}}[X]) < 18\ve^{1/4}$. 
\end{thm}

The following is extracted from Proof of Theorem 2.8 in \cite{st}. 
We replicate the proof for completeness. 
We view $\ell_\infty X$ as the diagonal subalgebra of $\IB(\ell_2X)$. 

\begin{lem}\label{lem:lip}
If $u \in \IB(\ell_2X)$ has $\ve$-propa\-ga\-tion at most $R$ 
and $h\in\ell_\infty X$ is such that $0\le h \le1$ and 
$|h(x)-h(y)|\le \delta$ for $\dist(x,y)\le R$, 
then $\|[h,u]\|\le 4\delta\|u\| + 2\delta^{-1}\ve$.
\end{lem}
\begin{proof}
Put $E(n):=h^{-1}([\delta n,\delta (n+1)))$ for $n=0,\ldots,\lfloor\delta^{-1}\rfloor$ 
and observe that $\bigsqcup_n E(n)=X$. 
For $|m-n|>1$, one has $\dist(E(m),E(n))>R$ 
and so $\| u_{m,n}\|<\ve$ for $u_{m,n}:= 1_{E(m)}u1_{E(n)}$. 
Consider $g := \sum_n \delta n 1_{E(n)}\in\ell_\infty X$. 
Then $\| g - h \|\le \delta$ and 
\[
\| [h,u] \| \approx_{2\delta\|u\|} \| [g,u] \| = 
\| \sum_{m,n} \delta(m-n) u_{m,n}\| 
\le 2 \delta\|u\| + 2\lfloor\delta^{-1}\rfloor\ve,
\]
because $\|\sum_{m-n=k} u_{m,n}\|\le \|u\|$ for any $k$ and $\le\ve$ if $|k|>1$.
\end{proof}

\begin{proof}[Proof of Theorem~\ref{thm:sz}]
Let $u_0$ be a contraction with $\ve$-propa\-gation at most $R$.
We take $S>0$ and $\mu\colon X\to\Prob(X)$ as in the definition 
of property~A for $\delta:=\ve^{1/2}$ and $R$. 
We further take $T>0$ and $\nu\colon X\to\Prob X$ 
for $\delta$ and $S$. 
Then $f_z(x) := \nu_x(z)^{1/2}$ 
satisfy $\sum_z f_z(x)^2 =1$ for every $x$, 
$\supp f_z \subset \Ball(z,T)$, and 
$\sum_z |f_z(x)-f_z(y)|^2 < \delta$ whenever $\dist(x,y)\le S$.
We view $f_z\in\ell_\infty X\subset\IB(\ell_2X)$ and 
define a unital completely positive 
map $\Phi_\nu$ on $\IB(\ell_2X)$ by 
$\Phi_\nu(u) := \sum_z f_z u f_z \in \IC_{\mathrm{u}}^{2T}[X]$. 
The RHS is convergent in the strong operator topology. 
It suffices to show $\| u_0 - \Phi_\nu(u_0) \| \le 18\delta^{1/2}$.

We consider the probability space $\Omega:=\{\pm1\}^X$ 
with the uniform probability measure and the 
i.i.d.\ Rademacher random variables 
$r_z\colon\Omega\ni\omega\mapsto\omega_z\in\{\pm1\}$. 
The family $\{ r_z\}_{z\in X}$ is orthonormal in $L^2(\Omega)$. 
We define a ``random'' function $f_\omega\in\ell_\infty X$ by 
\[
f_\omega(x) := \sum_{z\in X} r_z(\omega) f_z(x) = \sum_{z\in \Ball(x,T)} \omega_z f_z(x).
\] 
Observe that 
$\| f_\omega\|_\infty \le N_X(T)^{1/2}<\infty$, 
$\int f_\omega(x)^2\,d\omega=1$ for every $x$, 
and  
\[
\int |f_\omega(x) - f_\omega(y)|^2\,d\omega
 = \sum_z |f_z(x)-f_z(y)|^2 < \delta
\]
whenever $\dist(x,y)\le S$. 
Moreover, for every $u\in\IB(\ell_2X)$, one has 
\[
\int f_\omega u f_\omega\,d\omega  = \sum_z f_z u f_z = \Phi_\nu(u). 
\]
We will perturb $f$ inside $\ell_\infty(X;L^2(\Omega))$. 
The perturbation is matched with the Cauchy--Schwarz inequality 
that any strong operator topology measurable 
operator-valued random variables $a$ and $b$ satisfy  
\[
\| \int a(\omega)^*b(\omega)\,d\omega\|
 \le \|\int a(\omega)^*a(\omega)\,d\omega\|^{1/2} 
   \|\int b(\omega)^*b(\omega)\,d\omega\|^{1/2}. 
\]

Put $C:=\delta^{-1/2}$ and $g_\omega(x) := -C \vee f_\omega(x) \wedge C$ 
so that $\| g_\omega\|_\infty\le C$. 
Since 
\begin{align*}
\int f_\omega(x)^4\,d\omega
 &= \sum_{y,z,v,w}\int r_y(\omega)r_z(\omega)r_v(\omega)r_w(\omega) f_y(x)f_z(x)f_v(x)f_w(x)\,d\omega\\
 &=3\sum_{z,w;\ z\neq w}f_z(x)^2f_w(x)^2 + \sum_z  f_z(x)^4\\
 &=3 - 2 \sum_z  f_z(x)^4 \le 3,
\end{align*}
one has for every $x$ 
\[
\int |f_\omega(x)-g_\omega(x)|^2 \,d\omega
 \le C^{-2}\int f_\omega(x)^4\,d\omega \le 3C^{-2}.
\]
Put $h_\omega(x) := \sum_z \mu_x(z)g_\omega(z)$. 
Then one has $\| h_\omega\|_\infty\le C$, 
\begin{align*}
\int |g_\omega(x)-h_\omega(x)|^2\,d\omega
 &= \int |\sum_z \mu_x(z) (g_\omega(x)-g_\omega(z))|^2\,d\omega\\
 &\le \sum_z \mu_x(z) \int |g_\omega(x)-g_\omega(z)|^2\,d\omega
 < \delta
\end{align*}
since $\supp\mu_x \subset\Ball(x,S)$ and 
$|g_\omega(x)-g_\omega(z)|\le|f_\omega(x)-f_\omega(z)|$, and 
\begin{align*}
|h_\omega(x) - h_\omega(y)| &\le \| \mu_x-\mu_y \|_1\|g_\omega\|_\infty\le C\delta
\end{align*}
for every $\omega$ and every $x$ and $y$ 
such that $\dist(x,y)\le R$. 
By Lemma~\ref{lem:lip} applied to $(2C)^{-1}(h_\omega+C)$, 
one has
\[
\|[h_\omega,u_0]\|
 \le (2\delta + 4\delta^{-1}\ve)\cdot 2C \le 12\delta^{1/2}
\]
for every $\omega$.
Consequently, by the Cauchy--Schwarz inequality, one has
\begin{align*}
\| u_0 - \Phi_\nu(u_0) \| &= \|\int f_\omega[f_\omega,u_0]\,d\omega\|\\
 &\le \|\int f_\omega [h_\omega,u_0]\,d\omega\|+2(3^{1/2}C^{-1}+\delta^{1/2})\\
 &\le 12\delta^{1/2} + 6\delta^{1/2}.
\end{align*}
This completes the proof. 
\end{proof}

\end{document}